\documentclass[12pt,reqno]{amsart}

\usepackage{url}
\usepackage{pdfsync}

\usepackage{amssymb}

\usepackage{mathrsfs}

\DeclareMathAlphabet{\mathpzc}{OT1}{pzc}{m}{it}

\usepackage{tikz}
\usetikzlibrary{arrows,decorations.pathmorphing,decorations.pathreplacing,positioning,shapes.geometric,shapes.misc,decorations.markings,decorations.fractals,calc,patterns}

\usepackage{graphicx}

\def\cA{\mathscr{A}}

\def\fA{\mathfrak{A}}

\def\fa{\mathfrak{a}}
\def\fb{\mathfrak{b}}
\def\fc{\mathfrak{c}}
\def\fd{\mathfrak{d}}

\def\sA{\mathsf{A}}

\def\sC{\mathsf{C}}

\def\sT{\mathsf{T}}
\def\sU{\mathsf{U}}

\def\sX{\mathsf{X}}
\def\sY{\mathsf{Y}}

\def\add{\operatorname{add}}

\def\adots{\mathinner{\mkern1mu\raise1.0pt\vbox{\kern7.0pt\hbox{.}}\mkern2mu\raise4.0pt\hbox{.}\mkern2mu\raise7.0pt\hbox{.}\mkern1mu}}

\def\dim{\operatorname{dim}}

\def\Ext{\operatorname{Ext}}

\def\nc{\operatorname{nc}}

\def\prod{\operatorname{prod}}

\def\P{\mathcal P}%
\def\F{\mathcal F}%
\def\G{\mathcal G}%

\newcommand{\Dfn}[1]{\emph{#1}} 

\numberwithin{equation}{section}

\newtheorem{Lemma}{Lemma}[section]

\newtheorem{Proposition}[Lemma]{Proposition}

\theoremstyle{definition}
\newtheorem{Definition}[Lemma]{Definition}

\newtheorem{Remark}[Lemma]{Remark}

\begin{document}

\setlength{\parindent}{0pt}
\setlength{\parskip}{7pt}

\title[Ptolemy diagrams and torsion pairs]{Ptolemy diagrams and torsion pairs in the cluster category of Dynkin type $A_n$}

\author{Thorsten Holm}
\address{Institut f\"{u}r Algebra, Zahlentheorie und Diskrete
  Mathematik, Fa\-kul\-t\"{a}t f\"{u}r Ma\-the\-ma\-tik und Physik, Leibniz
  Universit\"{a}t Hannover, Welfengarten 1, 30167 Hannover, Germany}
\email{holm@math.uni-hannover.de}
\urladdr{http://www.iazd.uni-hannover.de/\~{ }tholm}

\author{Peter J\o rgensen}
\address{School of Mathematics and Statistics,
Newcastle University, Newcastle upon Tyne NE1 7RU, United Kingdom}
\email{peter.jorgensen@ncl.ac.uk}
\urladdr{http://www.staff.ncl.ac.uk/peter.jorgensen}

\author{Martin Rubey}
\address{Institut f\"{u}r Algebra, Zahlentheorie und Diskrete
  Mathematik, Fa\-kul\-t\"{a}t f\"{u}r Ma\-the\-ma\-tik und Physik, Leibniz
  Universit\"{a}t Hannover, Welfengarten 1, 30167 Hannover, Germany}
\email{martin.rubey@math.uni-hannover.de}
\urladdr{http://www.iazd.uni-hannover.de/~rubey}


\thanks{{\em Acknowledgement. }This work has been carried out in the framework of the
  research priority programme SPP 1388 {\em Darstellungstheorie} of
  the Deutsche Forschungsgemeinschaft (DFG).  We gratefully acknowledge
  financial support through the grant HO 1880/4-1. }



\begin{abstract} 

  We give a complete classification of torsion pairs in the cluster
  category of Dynkin type $A_n$.  Along the way we give a new
  combinatorial description of Ptolemy diagrams, an infinite version
  of which was introduced by Ng in \cite{Ng}.  This allows us to count
  the number of torsion pairs in the cluster category of type $A_n$.
  We also count torsion pairs up to Auslander-Reiten translation.

\end{abstract}

\maketitle

\setcounter{section}{0}
\section{Introduction}
\label{sec:introduction}

Let $\cA$ be the cluster algebra of Dynkin type $A_n$, let $\sC$ be
the cluster category of Dynkin type $A_n$, and let $P$ be a (regular)
$(n+3)$-gon.  There are bijections between the following sets:
\begin{enumerate}

  \item  Clusters in $\cA$,

\smallskip

  \item  Cluster tilting objects in $\sC$,

\smallskip

  \item  Triangulations by non-crossing diagonals of $P$.

\end{enumerate}
See Caldero, Chapoton, and Schiffler \cite{CCS} and Iyama \cite{Iyama1}.

To place this in a larger context, note that if $u$ is a cluster
tilting object in $\sC$ and $\sU = \add(u)$ is the full subcategory
consisting of direct sums of direct summands of $u$, then $(\sU,\Sigma
\sU)$ is a so-called torsion pair by Keller and Reiten \cite[sec.\
2.1, prop.]{KellerReiten2}.  Here $\Sigma$ is the suspension functor
of the triangulated category $\sC$.  The triangulation on $\sC$ is due
to Keller \cite{Keller} and is based on the definition of $\sC$ as an
orbit category by Buan, Marsh, Reineke, Reiten, and Todorov
\cite{BMRRT}.

In this paper we widen the perspective by investigating general
torsion pairs in $\sC$.  A \Dfn{torsion pair} in a triangulated
category $\sT$ is a pair $(\sX,\sY)$ of full subcategories closed
under direct sums and direct summands such that
\begin{enumerate}

  \item the morphism space $\sT(x,y)$ is zero for $x \in \sX$, $y \in
  \sY$,

\smallskip

  \item each $t \in \sT$ sits in a distinguished triangle $x \rightarrow
  t \rightarrow y \rightarrow \Sigma x$ with $x \in \sX$, $y \in \sY$.

\end{enumerate}
This concept was introduced by Iyama and Yoshino in \cite[def.\
2.2]{IY}.  It is a triangulated version of the classical notion of a
torsion pair in an abelian category due to Dickson, see
\cite{Dickson}.  In the triangulated situation it has precursors in
the form of the t-structures of Beilinson, Bernstein, and Deligne,
where, additionally, one assumes $\Sigma \sX \subseteq \sX$ (see
\cite{BBD}), and the co-t-structures of Bondarko and Pauksztello
where, additionally, one assumes $\Sigma^{-1}\sX \subseteq \sX$ (see
\cite{Bondarko}, \cite{Pauksztello}).  Note that the terminology of
torsion pairs in triangulated categories was also employed by
Beligiannis and Reiten in \cite{BR}, but they used it as a synonym for
t-structures.

There has so far been little systematic investigation of torsion pairs
in triangulated categories, but Ng \cite{Ng} gave a complete
classification of torsion pairs in the cluster category of type
$A_{\infty}$ in terms of certain infinite combinatorial objects.  See
\cite{HJ} for details on this category.  In particular, Ng introduced
the Ptolemy condition which, when supplanted to the finite situation,
takes the following form: a Ptolemy diagram is a set of diagonals of a
finite polygon (with a distinguished oriented base edge) such that, if
the set contains crossing diagonals $\fa$ and $\fb$, then it contains
all diagonals which connect end points of $\fa$ and $\fb$.  See Figure
\ref{fig:Ptolemy} and Definition \ref{def:Ptolemy} below.
\begin{figure}
\[
  \begin{tikzpicture}[auto]
    \node[name=s, shape=regular polygon, regular polygon sides=20, minimum size=6cm, draw] {}; 
    \draw[thick] (s.corner 5) to node[very near start,below=20pt] {$\fa$} (s.corner 16);
    \draw[shift=(s.corner 5)] node[left] {$\alpha_1$};
    \draw[shift=(s.corner 16)] node[right] {$\alpha_2$};
    \draw[thick] (s.corner 9) to node[near start] {$\fb$} (s.corner 19);
    \draw[shift=(s.corner 9)] node[left] {$\beta_1$};
    \draw[shift=(s.corner 19)] node[right] {$\beta_2$};
    \draw[thick,dotted] (s.corner 5) to (s.corner 9);
    \draw[thick,dotted] (s.corner 5) to (s.corner 19);
    \draw[thick,dotted] (s.corner 16) to (s.corner 9);
    \draw[thick,dotted] (s.corner 16) to (s.corner 19);
  \end{tikzpicture} 
\]
\caption{The Ptolemy condition}
\label{fig:Ptolemy}
\end{figure}
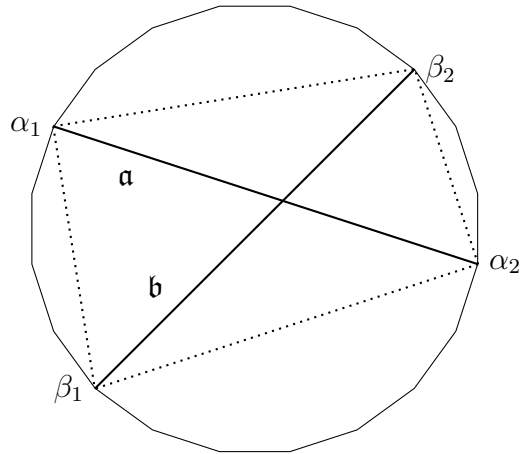

For instance, a polygon with no diagonals (an ``empty cell'') is a
Ptolemy diagram, as is a polygon with all diagonals (a ``clique'').
The triangle is the only Ptolemy diagram which is both an empty cell and
a clique.  If $A$ and $B$ are boundary edges of two Ptolemy diagrams,
then there is an obvious way of gluing $A$ to $B$ to obtain a new
Ptolemy diagram.  We will show the following classification result on
Ptolemy diagrams and torsion pairs.

{\bf Theorem A. }
{\em
\begin{enumerate}

  \item  There is a bijection between Ptolemy diagrams of the $(n+3)$-gon
  and torsion pairs in the cluster category $\sC$ of Dynkin type $A_n$.

\smallskip

  \item  Each Ptolemy diagram can be obtained by gluing empty cells and
    cliques.

\end{enumerate}
}

Note that a triangulation by non-crossing diagonals is a Ptolemy
diagram.  Under the bijection of part (i), it corresponds to a torsion
pair coming from a cluster tilting object.

Part (i) is a type $A_n$ analogue of Ng's classification, but our
proof is easier than hers because it uses the gluing in part (ii).
The gluing follows from the observation that if a diagonal in a
Ptolemy diagram crosses no other diagonal in the diagram, then it
divides the diagram into two smaller Ptolemy diagrams.  In fact, the
gluing can be organised so as to be unique, and this permits us to
prove the following counting result which, by virtue of part (i), also
counts torsion pairs in $\sC$.

{\bf Theorem B. }
{\em 
The number of Ptolemy diagrams of the $(n+3)$-gon is
$$
\frac{1}{n+2}\sum_{\ell\geq 0}%
2^\ell\binom{n+1+\ell}{\ell}\binom{2n+2}{n+1-2\ell}
$$
with the convention that the second binomial coefficient is $0$ for
$n+1-2\ell < 0$. 
}

The first few values, starting at $n = 0$, are
\begin{align*}
  &1, 4, 17, 82, 422, 2274, 12665, 72326, 421214, 2492112,\\
  &14937210,  90508256, 553492552, 3411758334, 21175624713, \\
  &132226234854, 830077057878,\dots
\end{align*}
This sequence may not have appeared previously in the literature.
Based on this paper, it is now item A181517 in the Online Encyclopedia
of Integer Sequences \cite{OEIS}.  Its asymptotic behaviour can be
determined explicitly, see Remark \ref{rmk:asymptotics}. 

We are also able to determine the generating function for Ptolemy
diagrams up to rotation, see Proposition \ref{pro:Martin}.  This
corresponds to counting torsion pairs up to Auslander-Reiten
translation.  The first few values are
\begin{align*}
  &1, 3, 5, 19, 62, 301, 1413, 7304, 38294, 208052, \\
  &1149018, 6466761, 36899604, 213245389, 1245624985, \\
  &7345962126, 43688266206,\dots
\end{align*}
Again it seems that this sequence was not encountered before.  It is
now item A181519 in the Online Encyclopedia of Integer Sequences.

K\"{o}hler \cite{Koehler} recently classified and counted thick
subcategories of triangulated ca\-te\-go\-ri\-es with finitely many
indecomposables.  This is the same as counting torsion pairs
$(\sX,\sY)$ in which $\sX$ and $\sY$ are triangulated subcategories;
these are known as stable t-structures.  One can show that the only
stable t-structures in the cluster category $\sC$ are $(\sC,0)$ and
$(0,\sC)$, so our results do not overlap with K\"{o}hler's.

\section*{Acknowledgement}
We are grateful to Christian Krattenthaler for a very useful
suggestion leading to the present formula in Theorem B which much
improves a previous version.  We also thank an anonymous referee
for reading the paper very carefully, correcting an error, and
suggesting several improvements to the presentation.  The diagrams
were typeset with TikZ.

\section{Characterizing torsion pairs combinatorially}

Let $P$ be an $(n+3)$-gon with a distinguished oriented edge which we
refer to as the \Dfn{distinguished base edge}.  We denote vertices
of the polygon by lower case Greek letters.  An \Dfn{edge} is a set of
two neighbouring vertices of the polygon.  A \Dfn{diagonal} is a set
of non-neighbouring vertices.  Two diagonals $\{\alpha_1,\alpha_2\}$
and $\{\beta_1,\beta_2\}$ \Dfn{cross} if their end points are all
distinct and come in the order $\alpha_1, \beta_1, \alpha_2, \beta_2$
when moving around the polygon in one direction or the other.  This
corresponds to an obvious notion of geometrical crossing.  Note that a
diagonal does not cross itself and that two diagonals sharing an end
point do not cross.

We recall the following from the introduction. 

\begin{Definition}
\label{def:Ptolemy}
Let $\fA$ be a set of diagonals in $P$.  Then $\fA$ is a \Dfn{Ptolemy
  diagram} if it has the following property: when $\fa = \{
\alpha_1,\alpha_2 \}$ and $\fb = \{ \beta_1,\beta_2 \}$ are crossing
diagonals in $\fA$, then those of $\{ \alpha_1,\beta_1\}$, $\{
\alpha_1,\beta_2 \}$, $\{ \alpha_2,\beta_1 \}$, $\{ \alpha_2,\beta_2
\}$ which are diagonals are in $\fA$.  See Figure \ref{fig:Ptolemy}.
\end{Definition}

Note that, because of the distinguished base edge which we draw in
bold, the two Ptolemy diagrams in Figure \ref{fig:squares} are
distinct.
\begin{figure}
\[
  \begin{tikzpicture}[auto]
    \node[name=s, shape=regular polygon, regular polygon sides=4, minimum size=5cm, draw] at (0,0) {}; 
    \draw[thick] (s.corner 1) to (s.corner 3);
    \draw[ultra thick,black] (s.corner 3) to (s.corner 4);

    \node[name=s, shape=regular polygon, regular polygon sides=4, minimum size=5cm, draw] at (5,0) {}; 
    \draw[thick] (s.corner 2) to (s.corner 4);
    \draw[ultra thick,black] (s.corner 3) to (s.corner 4);
  \end{tikzpicture} 
\]
\caption{Two different Ptolemy diagrams}
\label{fig:squares}
\end{figure}
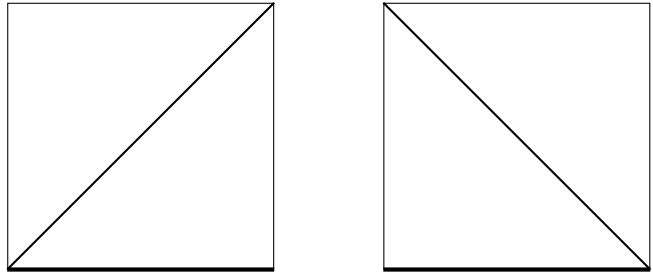

Let $\sC$ be the cluster category of type $A_n$.  There is a bijection
between indecomposable objects of $\sC$ and diagonals of $P$.  We use
lower case roman letters for (indecomposable) objects of $\sC$ and
lower case fraktur letters for the corresponding diagonals.  The
suspension functor $\Sigma$ acts on (indecomposable) objects and hence
on diagonals; the action on diagonals is rotation by one vertex.  Note
that $\Sigma$ is equal to the Auslander-Reiten translation of $\sC$
since $\sC$ is $2$-Calabi-Yau.  We have
\begin{equation}
\label{equ:Ext}
  \dim \Ext_{\sC}^1(a,b) =
  \left\{
    \begin{array}{cl}
      1 & \mbox{if $\fa$ and $\fb$ cross}, \\
      0 & \mbox{otherwise,}
    \end{array}
  \right.
\end{equation}
see \cite{CCS}.

The bijection between indecomposable objects of $\sC$ and diagonals of
$P$ extends to a bijection between subcategories of $\sC$ closed under
direct sums and direct summands, and sets of diagonals of $P$.  We use
upper case sans serif letters for subcategories and upper case fraktur
letters for the corresponding sets of diagonals.  The suspension
functor acts on diagonals and hence on sets of diagonals.

\begin{Definition}
If $\fA$ is a set of diagonals, then
\[
  \nc \fA = \{\, \mbox{$\fb$ is a diagonal of $P$}
                 \,\mid\, \mbox{$\fb$ crosses no diagonal in $\fA$} \,\}.
\]
\end{Definition}

Figure \ref{fig:blue_and_red} is an example where $\fA$ consists of
the solid diagonals and $\nc \fA$ of the dotted ones.  Note that this is
not a Ptolemy diagram.
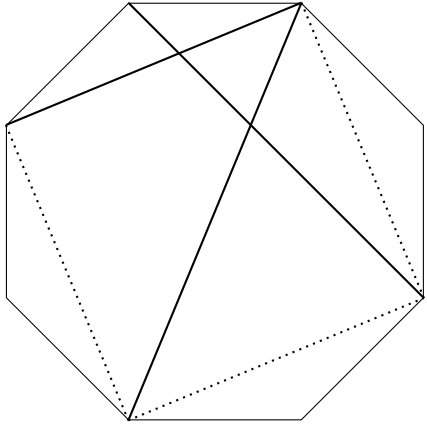
\begin{figure}
\[
  \begin{tikzpicture}[auto]
    \node[name=s, shape=regular polygon, regular polygon sides=8, minimum size=6cm, draw] {}; 
    \draw[thick] (s.corner 1) to (s.corner 3);
    \draw[thick] (s.corner 1) to (s.corner 5);
    \draw[thick] (s.corner 2) to (s.corner 7);
    \draw[thick,dotted] (s.corner 1) to (s.corner 7);
    \draw[thick,dotted] (s.corner 3) to (s.corner 5);
    \draw[thick,dotted] (s.corner 5) to (s.corner 7);
  \end{tikzpicture} 
\]
\caption{The dotted diagonals are $\nc$ of the solid diagonals}
\label{fig:blue_and_red}
\end{figure}
In the example, $\fA$ and $\nc \fA$ are disjoint but this is not
always the case since a diagonal does not cross itself.

Let $\sA$ be a subcategory of $\sC$ closed under direct sums and
direct summands.  We define the perpendicular subcategories by
\begin{align*}
  {}^{\perp}\sA & = \{\, c \in \sC \,|\, \sC(c,a) = 0 \;\mbox{for each}\; a \in \sA \,\}, \\
  \sA^{\perp} & = \{\, c \in \sC \,|\, \sC(a,c) = 0 \;\mbox{for each}\; a \in \sA \,\}.
\end{align*}
If $\sA$ corresponds to the set of diagonals $\fA$, then Equation
\eqref{equ:Ext} implies that ${}^{\perp}\sA$ corresponds to
$\Sigma^{-1}\nc \fA$ and $\sA^{\perp}$ corresponds to $\Sigma \nc
\fA$; this follows using $\sC(c,\Sigma d) = \Ext^1_{\sC}(c,d)$.  Note
that the operator $\nc$ commutes with $\Sigma$ and $\Sigma^{-1}$.

\begin{Proposition}
\label{pro:A}
The following are equivalent for a subcategory $\sA$ of $\sC$ which is
closed under direct sums and direct summands.
\begin{enumerate}

  \item  $\sA$ is closed under extensions, that is, if $a_1 ,a_2 \in
    \sA$ and $a_1 \rightarrow b \rightarrow a_2 \rightarrow
    \Sigma a_1$ is a distinguished triangle of $\sC$, then $b \in
    \sA$. 

\smallskip

  \item  $(\sA,\sA^{\perp})$ is a torsion pair.

\smallskip

  \item  $\sA = {}^{\perp}(\sA^{\perp})$.

\smallskip

  \item $\fA = \nc \nc \fA$.

\end{enumerate}
\end{Proposition}

\begin{proof}
(i)$\Rightarrow$(ii) holds by \cite[prop.\ 2.3(1)]{IY} since $\sA$
is contravariantly finite because it has only finitely many
indecomposable objects.  (Indeed, $\sC$ itself has only finitely many
indecomposable objects.)

(ii)$\Rightarrow$(iii) holds by the remarks following \cite[def.\
2.2]{IY}. 

(iii)$\Rightarrow$(i): If $\sX$ is any full subcategory of $\sC$ then
${}^{\perp}\sX$ is closed under extensions.  Namely, if $a_1, a_2 \in
{}^{\perp}\sX$ and $a_1 \rightarrow b \rightarrow a_2 \rightarrow
\Sigma a_1$ is a distinguished triangle, then each $x \in \sX$ gives
an exact sequence $\sC( a_2 , x ) \rightarrow \sC( b , x ) \rightarrow
\sC( a_1 , x)$.  The outer terms are $0$, so $\sC( b , x ) = 0$ whence
$b \in {}^{\perp}\sX$.

(iii)$\Leftrightarrow$(iv) follows from the remarks before the proposition
by which $\sA^{\perp}$ corresponds to $\Sigma \nc \fA$ and
${}^{\perp}(\sA^{\perp})$ corresponds to $\Sigma^{-1} \nc(\Sigma \nc
\fA) = \nc \nc \fA$.
\end{proof}

\begin{Remark}
\label{rmk:IY}
Note that by an easy argument, in a torsion pair $(\sX,\sY)$ we always
have $\sY = \sX^{\perp}$; see \cite[def.\ 2.2]{IY}.  It follows that
every torsion pair in $\sC$ has the form $(\sA,\sA^{\perp})$ for one
of the subcategories $\sA$ in Proposition \ref{pro:A}.  By the
proposition, there is hence a bijection between torsion pairs in $\sC$
and sets of diagonals $\fA$ with $\fA = \nc \nc \fA$.
\end{Remark}

Let $\P$ be the set of Ptolemy diagrams in polygons of any size with a
distinguished base edge.  For convenience, we will consider the edges
of the polygon to be part of a Ptolemy diagram.  Moreover, $\P$
includes the degenerate Ptolemy diagram consisting of two vertices and
the distinguished base edge.  We give a different (global) description
of Ptolemy diagrams by establishing a recursive combinatorial equation
for $\P$.

Recall that a polygon with no diagonals is called an \Dfn{empty cell}
and that a polygon with all diagonals is called a \Dfn{clique}; these
are both Ptolemy diagrams.

\begin{Proposition}
\label{pro:decomposition}
The set $\P$ is recursively given as the disjoint union of
\begin{enumerate}

\item the degenerate Ptolemy diagram,

\smallskip

\item an empty cell with at least three edges, one of which is the
  distinguished base edge, where we have glued onto each other edge
  an element of $\P$ along its distinguished base edge,

\smallskip

\item a clique with at least four edges, one of which is the
  distinguished base edge, where we have glued onto each other
  edge an element of $\P$ along its distinguished base edge.

\end{enumerate}
These types correspond to the three parts of the right hand side of
the equation in Figure~\ref{fig:decomposition}.  In particular, a
Ptolemy diagram can be decomposed completely into Ptolemy diagrams
which are either empty cells or cliques.
\end{Proposition}

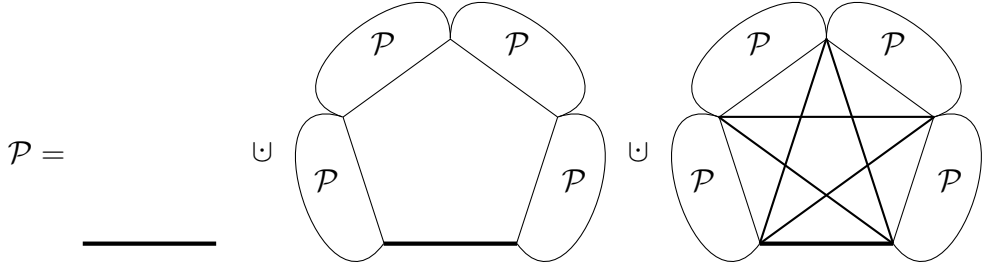
\begin{figure}
 \centering
  \begin{tikzpicture}[auto]
    \node[name=3] at (-0.5,0) {$\P = $};

    \node[name=s, shape=regular polygon, regular polygon sides=5, minimum size=3cm] at (1,0) {}; 
    \draw[ultra thick,black] (s.corner 3) to (s.corner 4);

    \node[name=s] at (2.5,0) {$\ensuremath{\mathaccent\cdot\cup}$};

    \node[name=s, shape=regular polygon, regular polygon sides=5, minimum size=3cm, draw] at (5,0) {}; 
    \draw[ultra thick,black] (s.corner 3) to (s.corner 4);
    \draw[out=306,in=18,looseness=2] (s.corner 4) to (s.corner 5);
    \draw[shift=(s.side 4)] node[right=5pt] {$\P$};
    \draw[out=18,in=90,looseness=2] (s.corner 5) to (s.corner 1);
    \draw[shift=(s.side 5)] node[above=5pt] {$\;\;\;\P$};
    \draw[out=90,in=162,looseness=2] (s.corner 1) to (s.corner 2);
    \draw[shift=(s.side 1)] node[above=5pt] {$\P\;\;\;$};
    \draw[out=162,in=234,looseness=2] (s.corner 2) to (s.corner 3);
    \draw[shift=(s.side 2)] node[left=5pt] {$\P$};

    \node[name=s] at (7.5,0) {$\ensuremath{\mathaccent\cdot\cup}$};

    \node[name=s, shape=regular polygon, regular polygon sides=5, minimum size=3cm, draw] at (10,0) {}; 
    \draw[ultra thick,black] (s.corner 3) to (s.corner 4);
    \draw[out=306,in=18,looseness=2] (s.corner 4) to (s.corner 5);
    \draw[shift=(s.side 4)] node[right=5pt] {$\P$};
    \draw[out=18,in=90,looseness=2] (s.corner 5) to (s.corner 1);
    \draw[shift=(s.side 5)] node[above=5pt] {$\;\;\;\P$};
    \draw[out=90,in=162,looseness=2] (s.corner 1) to (s.corner 2);
    \draw[shift=(s.side 1)] node[above=5pt] {$\P\;\;\;$};
    \draw[out=162,in=234,looseness=2] (s.corner 2) to (s.corner 3);
    \draw[shift=(s.side 2)] node[left=5pt] {$\P$};
    \draw[thick] (s.corner 1) to (s.corner 3);
    \draw[thick] (s.corner 1) to (s.corner 4);
    \draw[thick] (s.corner 2) to (s.corner 4);
    \draw[thick] (s.corner 2) to (s.corner 5);
    \draw[thick] (s.corner 3) to (s.corner 5);

  \end{tikzpicture} 
  \caption{The decomposition of the set of Ptolemy diagrams with
    a distinguished base edge.}
  \label{fig:decomposition}
\end{figure}

\begin{proof}
  It is clear that the sets (i), (ii), and (iii) are disjoint.

  Let a non-degenerate Ptolemy diagram $\fA$ be given with
  distinguished base edge $\{ \alpha,\beta \}$.  We will show that
  $\fA$ is either of type (ii) or type (iii).  For convenience, we
  will consider the vertices of the polygon to be ordered in an obvious
  way, starting with $\alpha$ and ending with $\beta$.

  Type (ii): Suppose that there do not exist crossing diagonals $\fa$
  and $\fb$ in $\fA$ ending in $\alpha$, respectively $\beta$.  We
  will show that $\fA$ is of type (ii).

  Consider increasing sequences of vertices $\alpha$, $\gamma_1$,
  $\ldots$, $\gamma_m$, $\beta$ with $m \geq 1$ for which the edges
  and diagonals
\[
  \{ \alpha,\gamma_1 \}, \{ \gamma_1,\gamma_2 \}, 
  \ldots, \{ \gamma_{m-1},\gamma_m \}, \{ \gamma_m,\beta \} 
\]
  are in $\fA$, and choose a sequence with $m$ minimal.  For ease of
  notation write $\gamma_0 = \alpha$ and $\gamma_{m+1} = \beta$.  The
  displayed edges and diagonals along with the distinguished base edge
  $\{ \alpha, \beta \}$ bound a region $C$.

  We show that $\fA$ is of type (ii) by showing that no diagonal in
  $\fA$ intersects the interior of $C$: then $C$ is an empty cell and
  each $\{ \gamma_j, \gamma_{j+1} \}$ with $0 \leq j \leq m$ divides
  $C$ from a (smaller) Ptolemy diagram; see
  Figure~\ref{fig:decomposition}.  Note that each smaller Ptolemy
  diagram is clearly uniquely determined.

  Suppose that $\fA$ does contain a diagonal $\{ \epsilon_1,\epsilon_2
  \}$ intersecting the interior of $C$.  We can assume $\epsilon_1 <
  \epsilon_2$.  There are three cases, each leading to a
  contradiction.
\begin{enumerate}

  \item[(a)] $\epsilon_1$ and $\epsilon_2$ are among the $\gamma_i$.
  Then $\epsilon_1 = \gamma_{j-1}$ and $\epsilon_2 = \gamma_{k+1}$
  where $1 \leq j \leq k \leq m$.  This contradicts that $m$ is
  minimal.

\smallskip

  \item[(b)] One of $\epsilon_1$ and $\epsilon_2$ is among the
  $\gamma_i$ and the other is not, see Figure \ref{fig:empty}.  By
  symmetry we can assume $\epsilon_1 = \gamma_{j-1}$ and $\gamma_k <
  \epsilon_2 < \gamma_{k+1}$ with $1 \leq j \leq k \leq m$.  The
  diagonals $\{ \epsilon_1, \epsilon_2 \} = \{ \gamma_{j-1},
  \epsilon_2 \}$ and $\{ \gamma_k, \gamma_{k+1} \}$ cross.  By the
  Ptolemy condition $\fc = \{ \gamma_{j-1}, \gamma_{k+1} \}$ is in
  $\fA$.

\smallskip
\noindent
  If $\fc$ intersects the interior of $C$ then we are in case (a).  If
  it does not, then we must have $\gamma_{j-1} = \alpha$ and
  $\gamma_{k+1} = \beta$.  But then there are crossing diagonals $\fa =
  \{ \alpha, \epsilon_2 \} = \{ \gamma_{j-1}, \epsilon_2 \} = \{
  \epsilon_1, \epsilon_2 \}$ and $\fb = \{ \beta, \gamma_k \} = \{
  \gamma_k, \gamma_{k+1} \}$ ending in $\alpha$, respectively $\beta$,
  contradicting our assumption on $\fA$.

\smallskip

  \item[(c)] $\epsilon_1$ and $\epsilon_2$ are not among the
  $\gamma_i$, see Figure \ref{fig:empty}.  Then $\gamma_{j-1} <
  \epsilon_1 < \gamma_j$ and $\gamma_k < \epsilon_2 < \gamma_{k+1}$
  for some $1 \leq j \leq k \leq m$.  The diagonal $\{ \epsilon_1,
  \epsilon_2 \}$ crosses each of the diagonals $\{ \gamma_{j-1},
  \gamma_j \}$ and $\{ \gamma_k, \gamma_{k+1} \}$, so by the Ptolemy
  condition each of the diagonals $\{ \epsilon_1, \gamma_{k+1} \}$ and
  $\{ \gamma_{j-1}, \epsilon_2 \}$ is in $\fA$.  These diagonals
  cross, so by the Ptolemy condition $\fc = \{ \gamma_{j-1},
  \gamma_{k+1} \}$ is in $\fA$.  Now conclude the argument by using
  the second paragraph of (b).

\end{enumerate}
  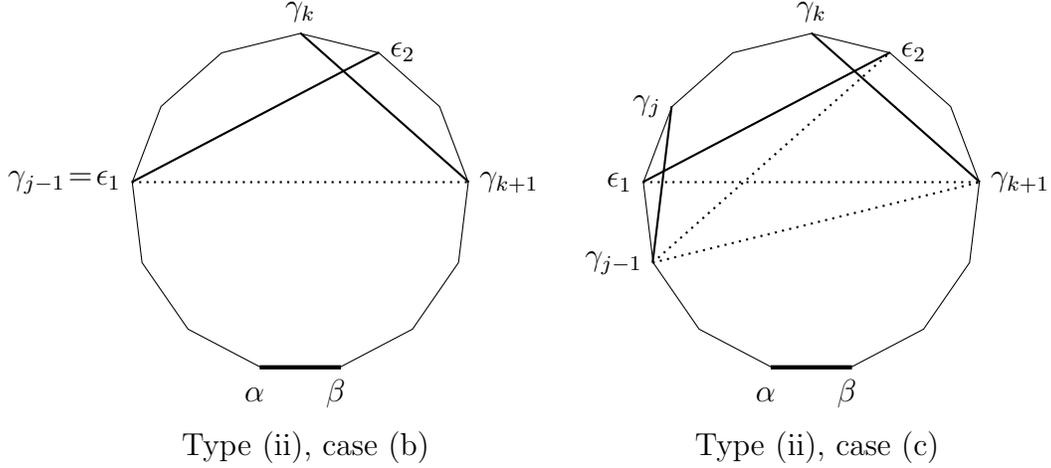
\begin{figure}
    \centering
    \begin{tabular}{cc}
      \begin{tikzpicture}[auto]
        \node[name=s, shape=regular polygon, regular polygon sides=13, minimum size=4.5cm, draw] {}; 
        \draw[shift=(s.corner 7)] node[below] {$\alpha$ \rule{0pt}{1.7ex}};
        \draw[shift=(s.corner 8)] node[below] {$\beta$ \rule{0pt}{1.7ex}};
        \draw[ultra thick] (s.corner 7) to (s.corner 8);
        \draw[shift=(s.corner 4)] node[left] {$\gamma_{j-1}\!=\!\epsilon_1$};
        \draw[shift=(s.corner 1)] node[above] {$\gamma_k$};
        \draw[shift=(s.corner 13)] node[right] {$\epsilon_2$};
        \draw[shift=(s.corner 11)] node[right] {$\gamma_{k+1}$};
        \draw[thick] (s.corner 4) to (s.corner 13);
        \draw[thick] (s.corner 1) to (s.corner 11);
        \draw[thick,dotted] (s.corner 4) to (s.corner 11);
      \end{tikzpicture} 
      &
      \begin{tikzpicture}[auto]
        \node[name=s, shape=regular polygon, regular polygon sides=13, minimum size=4.5cm, draw] {}; 
        \draw[shift=(s.corner 7)] node[below] {$\alpha$ \rule{0pt}{1.7ex}};
        \draw[shift=(s.corner 8)] node[below] {$\beta$ \rule{0pt}{1.7ex}};
        \draw[ultra thick] (s.corner 7) to (s.corner 8);
        \draw[shift=(s.corner 5)] node[left] {$\gamma_{j-1}$};
        \draw[shift=(s.corner 4)] node[left] {$\epsilon_1$};
        \draw[shift=(s.corner 3)] node[left] {$\gamma_j$};
        \draw[shift=(s.corner 1)] node[above] {$\gamma_k$};
        \draw[shift=(s.corner 13)] node[right] {$\epsilon_2$};
        \draw[shift=(s.corner 11)] node[right] {$\gamma_{k+1}$};
        \draw[thick] (s.corner 4) to (s.corner 13);
        \draw[thick] (s.corner 3) to (s.corner 5);
        \draw[thick] (s.corner 1) to (s.corner 11);
        \draw[thick,dotted] (s.corner 5) to (s.corner 11);
        \draw[thick,dotted] (s.corner 5) to (s.corner 13);
        \draw[thick,dotted] (s.corner 4) to (s.corner 11);
      \end{tikzpicture} \\
      \hspace{4ex} Type (ii), case (b) & Type (ii), case (c)
    \end{tabular}
    \caption{In type (ii), the diagonal $\{\epsilon_1, \epsilon_2\}$
    forces the presence of the diagonal $\{\gamma_{j-1},
    \gamma_{k+1}\}$.}
    \label{fig:empty}
  \end{figure}

  Type (iii):  Suppose that crossing diagonals $\fa = \{ \alpha,
  \delta \}$ and $\fb = \{ \beta, \delta' \}$ ending in $\alpha$,
  respectively $\beta$ do exist in $\fA$.  We will show that $\fA$ is
  of type (iii).

  By the Ptolemy condition $\{ \alpha, \delta' \}$ and $\{ \beta,
  \delta \}$ are in $\fA$.  Consider those vertices which are
  connected to each of $\alpha$ and $\beta$ by an edge or a diagonal
  in $\fA$.  Denote them by $\delta_1, \ldots, \delta_m$ in increasing
  order and note that $m \geq 2$ because $\delta$ and $\delta'$ are
  among the $\delta_i$.  For ease of notation write $\delta_0 =
  \alpha$ and $\delta_{m+1} = \beta$.

  Let $0 \leq j < k \leq m+1$.  Then $\{ \delta_j, \delta_k \}$ is in
  $\fA$.  Namely, this holds by definition if $j = 0$ since then
  $\delta_j = \alpha$.  So we can assume $1 \leq j$ and by symmetry $k
  \leq m$.  But then $\{ \alpha, \delta_k \}$ and $\{ \delta_j, \beta
  \}$ are crossing diagonals in $\fA$ and by the Ptolemy condition $\{
  \delta_j, \delta_k \}$ is in $\fA$.  So the $\delta_i$ form the
  vertices of a clique of edges and diagonals in $\fA$ which contains
  the distinguished base edge.

  We show that $\fA$ is of type (iii) by showing that if $\{ \delta_j,
  \delta_{j+1} \}$ is a diagonal with $0 \leq j \leq m$, then no
  diagonal in $\fA$ crosses $\{ \delta_j, \delta_{j+1} \}$: then $\{
  \delta_j, \delta_{j+1} \}$ divides the clique with vertices
  $\delta_i$ from a (smaller) Ptolemy diagram; see
  Figure~\ref{fig:decomposition}.  Note that, again, each smaller
  Ptolemy diagram is uniquely determined.

  So suppose that $\fA$ contains a diagonal $\{ \epsilon_1, \epsilon_2
  \}$ crossing $\{ \delta_j, \delta_{j+1} \}$.  We can assume that
  $\epsilon_1 < \epsilon_2$ and by symmetry considerations that
  $\delta_j < \epsilon_1 < \delta_{j+1}$.  Note that this entails $j
  \leq m-1$.

  There are two cases, each leading to a contradiction.
\begin{enumerate}

  \item[(a)] $\epsilon_2 \neq \beta$, see Figure
  \ref{fig:clique}.  Then the diagonal $\{ \epsilon_1, \epsilon_2 \}$
  crosses the diagonals $\{ \alpha, \delta_{j+1} \}$ and $\{ \beta,
  \delta_{j+1} \}$ so by the Ptolemy condition $\{ \alpha,
  \epsilon_1 \}$ and $ \{ \beta, \epsilon_1 \}$ are in $\fA$.
  Hence $\epsilon_1$ is among the $\delta_i$, contradicting 
  $\delta_j < \epsilon_1 < \delta_{j+1}$.

\smallskip

  \item[(b)] $\epsilon_2 = \beta$, see Figure \ref{fig:clique}.  Then
  $\{ \beta, \epsilon_1 \} = \{ \epsilon_1, \epsilon_2 \}$ is in
  $\fA$.  Moreover, $\{ \epsilon_1, \epsilon_2 \}$ crosses $\{ \alpha,
  \delta_{j+1} \}$ so by the Ptolemy condition $\{ \alpha, \epsilon_1
  \}$ is in $\fA$.  Hence $\epsilon_1$ is again among the $\delta_i$
  which is a contradiction.

\end{enumerate}
  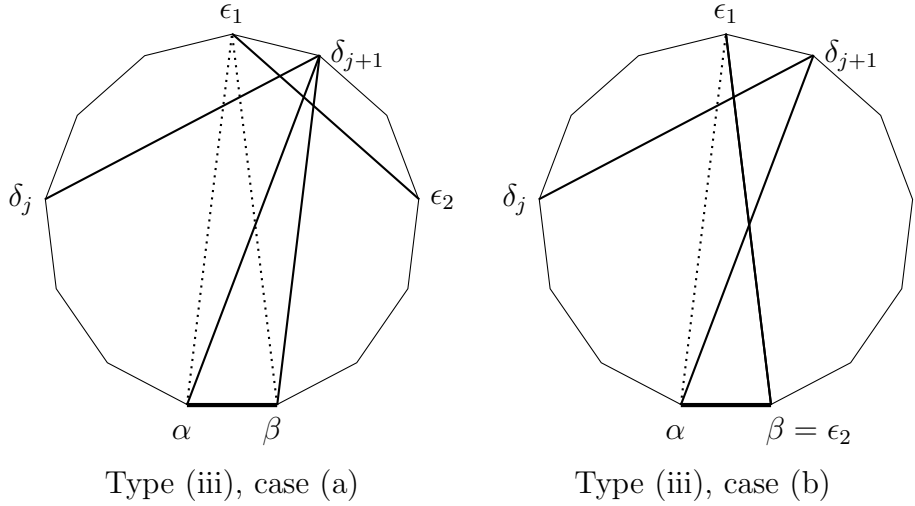
\begin{figure}
    \centering
    \begin{tabular}{cc}
      \begin{tikzpicture}[auto]
        \node[name=s, shape=regular polygon, regular polygon sides=13, minimum size=5cm, draw] {}; 
        \draw[shift=(s.corner 7)] node[below] {$\alpha$ \rule{0pt}{1.7ex}};
        \draw[shift=(s.corner 8)] node[below] {$\beta$ \rule{0pt}{1.7ex}};
        \draw[ultra thick] (s.corner 7) to (s.corner 8);
        \draw[shift=(s.corner 4)] node[left] {$\delta_j$};
        \draw[shift=(s.corner 1)] node[above] {$\epsilon_1$};
        \draw[shift=(s.corner 13)] node[right] {$\delta_{j+1}$};
        \draw[shift=(s.corner 11)] node[right] {$\epsilon_2$};
        \draw[thick] (s.corner 4) to (s.corner 13);
        \draw[thick] (s.corner 7) to (s.corner 13);
        \draw[thick] (s.corner 8) to (s.corner 13);
        \draw[thick] (s.corner 1) to (s.corner 11);
        \draw[thick,dotted] (s.corner 7) to (s.corner 1);
        \draw[thick,dotted] (s.corner 8) to (s.corner 1);
      \end{tikzpicture} 
      &
      \begin{tikzpicture}[auto]
        \node[name=s, shape=regular polygon, regular polygon sides=13, minimum size=5cm, draw] {}; 
        \draw[shift=(s.corner 7)] node[below] {$\alpha$ \rule{0pt}{1.7ex}};
        \draw[shift=(s.corner 8)] node[below] {$\lefteqn{\beta = \epsilon_2}$ \rule{0pt}{1.7ex}};
        \draw[ultra thick] (s.corner 7) to (s.corner 8);
        \draw[shift=(s.corner 4)] node[left] {$\delta_j$};
        \draw[shift=(s.corner 1)] node[above] {$\epsilon_1$};
        \draw[shift=(s.corner 13)] node[right] {$\delta_{j+1}$};
        \draw[thick] (s.corner 4) to (s.corner 13);
        \draw[thick] (s.corner 7) to (s.corner 13);
        \draw[thick] (s.corner 1) to (s.corner 8);
        \draw[thick,dotted] (s.corner 7) to (s.corner 1);
        \draw[thick] (s.corner 8) to (s.corner 1);
      \end{tikzpicture} \\
      Type (iii), case (a) & Type (iii), case (b)
    \end{tabular}
    \caption{In type (iii), the diagonal $\{\epsilon_1, \epsilon_2\}$
    forces the presence of the diagonals $\{\alpha, \epsilon_1\}$ and
    $\{ \beta,\epsilon_1 \}$.}
    \label{fig:clique}
  \end{figure}
\end{proof}

\begin{Remark}
\label{rmk:decomposition}
The proposition proves Theorem A(ii) from the introduction: each
Ptolemy diagram can be uniquely decomposed into regions, each of which
is either an empty cell or a clique.

Moreover, let $\fA$ be a Ptolemy diagram.  To obtain $\nc \fA$ from
$\fA$, one replaces empty cells by cliques and vice versa in the
decomposition.

Namely, let $\fd$ be an arbitrary diagonal.  If $\fd$ separates two
regions of $\fA$, then $\fd$ is one of the diagonals along which two
smaller Ptolemy diagrams have been glued in the decomposition to form
$\fA$, so clearly $\fd$ crosses no diagonal of $\fA$, so $\fd \in \nc
\fA$.  If $\fd$ is an internal diagonal in a clique, then it crosses
some other internal diagonal which must be in $\fA$, so $\fd \not\in
\nc \fA$.  If $\fd$ is an internal diagonal in an empty cell, then it
crosses no diagonal of $\fA$, so $\fd \in \nc
\fA$.

Note that we have $\fA = \nc \fA$ if and only if $\fA$ is a
triangulation of the polygon, since a triangle is the only polygon
which is an empty cell and a clique simultaneously.
\end{Remark}

With the above decomposition, we can show the following alternative
characterization of Ptolemy diagrams.

\begin{Proposition}
\label{pro:ncnc}
We have $\fA = \nc \nc \fA$ if and only if $\fA$ is a Ptolemy
diagram. 
\end{Proposition}

\begin{proof}
Suppose that $\fA = \nc \nc \fA$.  In Figure \ref{fig:Ptolemy},
consider the diagonal $\{ \alpha_1,\beta_1 \}$.  The diagonals
crossing it are precisely the diagonals which connect a vertex on one
side of $\{ \alpha_1,\beta_1 \}$ with a vertex on the other side of
$\{ \alpha_1,\beta_1 \}$.  But each such diagonal intersects $\fa$ or
$\fb$ so is outside $\nc \fA$.  Hence $\{ \alpha_1,\beta_1 \}$ is in
$\nc \nc \fA = \fA$.  The other diagonals in the Ptolemy condition
follow similarly.

Conversely, suppose that $\fA$ satisfies the Ptolemy condition.  By
Remark \ref{rmk:decomposition}, the operator $\nc$ interchanges empty
cells and cliques in the decomposition of $\fA$ according to
Proposition \ref{pro:decomposition}, so it is clear that $\fA = \nc
\nc \fA$.
\end{proof}

\begin{Remark}
  Combining Remark \ref{rmk:IY} and Proposition \ref{pro:ncnc} proves
  Theorem A(i) of the introduction.  In particular, to count torsion
  pairs in the cluster category of type $A_n$ we only need to
  determine the number of Ptolemy diagrams of the $(n+3)$-gon with a
  distinguished base edge.
\end{Remark}

\section{Counting the number of Ptolemy diagrams}

In this section we deduce expressions for the number of Ptolemy
diagrams.  First we compute the number of Ptolemy diagrams with
a distinguished base edge.  In a second step, we also determine the
number of Ptolemy diagrams up to rotation.

\subsection{Ptolemy diagrams with a distinguished base edge}
\label{sec:diagrams-base-edge}

Using combinatorial reasoning we shall obtain below an equation for
the (ordinary) generating function
\begin{equation}
  \label{eq:Ptolemy-generating-function}
  \P(y)=\sum_{N\geq1} \#\{\text{Ptolemy diagrams of the $(N+1)$-gon}\}
  y^N.  
\end{equation}
Let us briefly recall some facts from the general theory of
generating functions, see for example the book by Bergeron, Labelle
and Leroux~\cite[sec.\ 1.3]{BLL} or Aigner \cite[secs.\ 3.2 and
3.3]{A}.  Of course, our objective is to convey the general idea,
precise formulations are given in the cited textbooks.

Let $\F$ and $\G$ be sets of objects.  Each object is assigned to a
non-negative integer, referred to as its size.  Let $\F(y)$ and
$\G(y)$ be their generating functions.  Then the generating function
\begin{itemize}
\item for the disjoint union of $\F$ and $\G$ is $\F(y)+\G(y)$, and
\item for the set of objects obtained by pairing objects from $\F$
  and $\G$ is $\F(y)\G(y)$, where the size of a pair is the sum of
  the sizes of its two components.
\end{itemize}
Because of the natural correspondence with the operation on
generating functions, we denote the pairing of sets considered in the
second item by $\F\cdot\G$.

We can now derive an equation for the generating function of lists of
Ptolemy diagrams $\mathcal L^{\mathcal P}$.  Namely, either such a
list is empty, or it is a pair whose first component is a Ptolemy
diagram and whose second component is a list of Ptolemy diagrams.  We
thus have
$$
\mathcal L^{\P} = \emptyset\quad {\mathaccent\cdot\cup}\quad \P \cdot
\mathcal L^{\P},
$$
or, on the level of generating functions
$$
\mathcal L^{\P}(y) = 1 + \P(y) \mathcal L^{\P}(y),
$$
which entails 
$$
\mathcal L^{\P}(y)=\frac{1}{1-\P(y)}.
$$

Clearly, we can interpret the set of Ptolemy diagrams of type (ii) in
Proposition \ref{pro:decomposition} as the set of lists of Ptolemy
diagrams with at least two elements.  With a slight shift of
perspective, this is the same as a triple, whose first two components
are Ptolemy diagrams, and whose last component is a list of diagrams.
Hence, this set has generating function $\P(y)^2/\big(1-\P(y)\big)$.
Similarly, a Ptolemy diagram of type (iii) in Proposition
\ref{pro:decomposition} can be interpreted as a list of diagrams with
at least three elements.  Namely, recall that in the decomposition of
Proposition~\ref{pro:decomposition}, the cliques which occur have at
least four edges, one of which is the distinguished base edge; to the
other three we can attach Ptolemy diagrams.

In summary, using the combinatorial decomposition of Proposition
\ref{pro:decomposition} sketched in Figure~\ref{fig:decomposition},
\begin{equation*}
  \P(y) = y + \frac{\P(y)^2}{1-\P(y)} +
  \frac{\P(y)^3}{1-\P(y)}.
\end{equation*}

Let us rewrite this equation (essentially multiplying by $1-\P(y)$),
to make it amenable to Lagrange inversion
(eg.~\cite[sec.\ 3.1]{BLL} or \cite[thm.\ 3.8]{A}):
\begin{equation*}
  \P(y)=y\frac{1-\P(y)}
  {1- 2 \P(y) - \P(y)^2},
\end{equation*}
i.e., $\P(y)=yA(\P(y))$ with $A(y)=(1-y)/(1-2y-y^2)$.
Thus, denoting the coefficient of $y^N$ in $\P(y)$ with
$[y^N] \P(y)$, we have
\begin{align*}
  [y^N] \P(y)=\frac{1}{N}[y^{N-1}] \left(\frac{1-y}{1-2y-y^2}\right)^N.
\end{align*}
We can now apply the binomial theorem $(1+z)^a=\sum_{k\geq0}
\binom{a}{k} z^k$, for $a\in \mathbb Z$ and
$\binom{a}{k}=a(a-1)\cdots(a-k+1)/k!$, to transform the right hand
side into a sum.  As pointed out by Christian Krattenthaler the
result becomes much nicer if we first rewrite the expression
slightly, taking advantage of the fact that $1-2y-y^2$ is \lq
almost\rq\ $(1-y)^2$:
\begin{align}
  \nonumber
  (1-y)^N(1-2y-y^2)^{-N}
  &=(1-y)^{-N}\left(1-\frac{2y^2}{(1-y)^2}\right)^{-N}\\
  \nonumber
  &=(1-y)^{-N}\sum_{\ell\geq 0}%
  \binom{-N}{\ell}(-1)^\ell\frac{(2y^2)^\ell}{(1-y)^{2\ell}}\\
  \nonumber
  &=%
  \sum_{\ell\geq 0}\binom{-N}{\ell}(-1)^\ell (2y^2)^\ell%
  \sum_{k\geq 0}\binom{-N-2\ell}{k}(-1)^k y^k\\
  \label{equ:Martin}
  &=\sum_{k,\ell\geq 0}%
  \binom{-N}{\ell}\binom{-N-2\ell}{k}(-1)^{k+\ell}2^\ell y^{k+2\ell}.
\end{align}
Extracting the coefficient of $y^{N-1}$ in
Equation~\eqref{equ:Martin} by setting $k=N-1-2\ell$ we obtain
\[%
[y^N] \P(y) =\frac{1}{N}\sum_{\ell\geq 0}%
\binom{-N}{\ell}\binom{-N-2\ell}{N-1-2\ell}(-1)^{N-1-\ell}2^{\ell}.
\]
Finally, using
$\binom{-N}{\ell}=(-1)^\ell\binom{N+\ell-1}{\ell}$, we get that
the number of Ptolemy diagrams of the $(N+1)$-gon with
a distinguished base edge is
$$
\frac{1}{N}\sum_{\ell\geq 0}%
2^\ell\binom{N-1+\ell}{\ell}\binom{2N-2}{N-1-2\ell}.
$$
Setting $N = n+2$ proves Theorem B of the introduction, and the first
few values are given there.
\begin{Remark}
  Note that Petkov{\v{s}}ek's algorithm {\bf hyper}~\cite[sec.\
  8]{AeqB} \emph{proves} that the sum above cannot be written as a
  linear combination of (a fixed number of) hypergeometric terms.
\end{Remark}
\begin{Remark}
\label{rmk:asymptotics}
  Since the generating function $\P(y)$ satisfies an algebraic
  equation, the asymptotic behaviour of the coefficients of $\P(y)$
  can be extracted automatically, for example using the {\tt
    equivalent} function in Bruno Salvy's package {\bf gdev}
  available at \url{http://algo.inria.fr/libraries/}.  Thus, we learn
  that the leading term of the asymptotic expansion of $[y^N] \P(y)$
  is
  $$
  \frac{\alpha}{\sqrt{\pi N^3}} \rho^N,
  $$
  where $\rho=6.847333996370022\dots$ is the largest positive root of
  $8x^3 - 48x^2 - 47x + 4$ and $\alpha=0.10070579427884086\dots$ is
  the smallest positive root of $1136x^6 - 71x^4 - 98x^2 + 1$.
\end{Remark}

\subsection{Ptolemy diagrams up to rotation}
\label{sec:diagrams-isotypes}

Let us now turn to the enumeration of Ptolemy diagrams up to
rotation.  It seems easiest to apply a relatively general technique
known as the \lq dissymmetry theorem for trees\rq.  Namely, we will
consider Ptolemy diagrams as certain planar trees, where each inner
vertex of the tree corresponds to either an empty cell or a clique of
the diagram.  Thus, we will have to count trees according to their
number of leaves, where the edges incident to an inner vertex are
cyclically ordered and additionally these inner vertices \lq know\rq\
whether they correspond to an empty cell or such a clique.  This
situation is covered by Proposition~\ref{pro:dissymmetry} below.

This proposition is phrased in the language of combinatorial species
(as described in~\cite{BLL}), which is at first a tool to compute with
\emph{labelled} objects.  Formally, a species is a \emph{functor} from
the category of finite sets with bijections into itself.  Thus,
applying a species $\F$ to a finite set $U$ -- namely, a set of
labels, we obtain a new set $\F[U]$ -- namely the set of objects that
can be produced using the given labels.  Applying $\F$ to a bijection
$\sigma: U\to V$ produces a bijection $\F[\sigma]: \F[U]\to \F[V]$,
which, by functoriality, corresponds to relabelling the objects.
(However, when defining a particular species here, we refrain from
giving a precise definition of this relabelling operation.)

A simple but nevertheless important species is the \Dfn{singleton
  species $Y$}: it returns the input set $U$ if $U$ has cardinality
one and otherwise the empty set.  Another basic species we will need
is the \Dfn{species of unordered pairs $E_2$}, which returns the
input set $U$ if $U$ has cardinality two and the empty set otherwise.
Finally, for $k\geq 1$ we introduce the \Dfn{species of cycles
  $C_k$}, which consists of all (oriented) cycles with $k$ labelled
vertices.

We associate to every species $\F$ a so called \Dfn{exponential
  generating function} $\F(y)$, which is given by
$$
\F(y)=\sum_{N\geq1} \# \F[\{1,2,\dots,N\}] \frac{y^N}{N!},
$$
i.e., the coefficient of $y^N$ is the number of objects with labels
$\{1,2,\dots,N\}$ produced by $\F$, divided by $N!$.  In particular,
the exponential generating function associated to $Y$ is $Y(y)=y$,
and the exponential generating function associated to $E_2$ is
$E_2(y)=y^2/2$.  Finally, $C_k(y)=(k-1)!\frac{y^k}{k!} =
\frac{y^k}{k}$.

There are natural definitions for the sum $\F+\G$, the product
$\F\cdot\G$ and the composition $\F\circ\G$ of two species $\F$ and
$\G$.  We only give informal descriptions of the sets of objects which
they produce, and refer for precise definitions to~\cite[sec.\
1]{BLL}.  Let $U$ be a set of labels, then
\begin{itemize}
\item the set of objects in $(\F+\G)[U]$ is the disjoint union of
  $\F[U]$ and $\G[U]$,
\item the set of objects in $(\F\cdot\G)[U]$ is obtained by
  partitioning the set $U$ in all possible ways into two disjoint
  (possibly empty) sets $V$ and $W$ such that $U=V\cup W$, and
  producing all pairs of objects in 
  $$
  (\F[V],\G[W]),
  $$
  i.e., $\{(f,g)\,\mid\, f\in\F[V],g\in\G[W]\}$,
\item the set of objects in $(\F\circ\G)[U]$ is the set of all tuples
  of the form
  $$
  \big(\F[\{1,2,\dots,k\}], \G[B_1], \G[B_2],\dots, \G[B_k]\big),
  $$
where $\{B_1, B_2,\dots,B_k\}$ is a set partition of $U$.
\end{itemize}
The composition of species can be visualised by taking an object
produced by $\F$, and replacing all its labels by objects produced by
$\G$, such that the set of labels is exactly $U$.  In particular,
$\F\circ Y=Y\circ\F=\F$.

Finally, we need to describe the derivative $\F'$ of a species $\F$.
Given a set of labels $U$, we set
$\F'[U]=\F[U\,{\mathaccent\cdot\cup}\,\{*\}]$, where $*$ is a \lq
transcendental\rq\ element, i.e., an element that does not appear in
$U$.

It should not come as a surprise (although it certainly needs a proof)
that the exponential generating functions associated to the sum, the
product, the composition, and the derivative of species are
respectively $\F(y)+\G(y)$, $\F(y)\cdot\G(y)$, $\F\big(\G(y)\big)$ and
$\F'(y)$.

It remains to introduce the species of $R$-enriched trees $b_R$ and
$R'$-enriched rooted trees $B_{R'}$ with labels on the leaves, see
\cite[def.\ 13, sec.\ 3.1 and pg.\ 287, sec.\ 4.1]{BLL}: let $R$ be a
species with $\# R[\emptyset] = 0$, $\# R[\{1\}]=1$ and $\#
R[\{1,2\}]=0$.  Then an $R$-enriched tree on a set of labels $U$ is a
tree with at least two vertices, whose vertices of degree one (i.e.,
the leaves) correspond to the labels in $U$.  Additionally, every
vertex is assigned an object from $R[N]$, where $N$ is the set of
neighbours of the vertex.  Since $\# R[\{1,2\}]=0$ there are no
vertices of degree two.  Therefore, any such tree must have more
leaves than inner vertices and thus the set of $R$-enriched trees
with a finite number of leaves is finite.  The condition $\#
R[\{1\}]=1$ implies that only the inner vertices carry additional
structure.

An $R'$-enriched rooted tree on a set of labels $U$ is a rooted tree,
possibly an isolated vertex, where the vertices of degree at most one
(i.e., the leaves) correspond to the labels in $U$.  Additionally,
every vertex is assigned an object from $R'[N]$, where $N$ is the set
of those neighbours of the vertex which are further away from the
root than the vertex itself.  Again, since $\# R'[\{1\}]=0$, no
vertex can have a single successor and thus the set of $R'$-enriched
rooted trees with a finite number of leaves is finite.

\begin{figure}
  \centering
  \begin{tikzpicture}[auto]


    \node[name=s, shape=regular polygon, regular polygon sides=8, minimum size=6cm, draw] {}; 

    \draw[shift=(s.corner 1)] node[above=5pt] {3};
    \draw[shift=(s.corner 2)] node[above=5pt] {2};
    \draw[shift=(s.corner 3)] node[left=5pt] {6};
    \draw[shift=(s.corner 4)] node[left=5pt] {5};
    \draw[shift=(s.corner 5)] node[below=7pt] {*};
    \draw[shift=(s.corner 5)] node[inner sep=-1pt,below=3pt] (x) {};
    \draw[shift=(s.corner 5)] node[inner sep=-1pt,below=1.5pt] (z) {};
    \draw[shift=(s.corner 6)] node[below=7pt] {7};
    \draw[shift=(s.corner 6)] node[inner sep=-1pt,below=3pt] (y) {};
    \draw[shift=(s.corner 7)] node[right=5pt] {1};
    \draw[shift=(s.corner 8)] node[right=5pt] {4};

%
%
    \draw[ultra thick,black] (s.corner 5) to (s.corner 6);

    \draw[thick] (s.corner 1) to (s.corner 7);
    \draw[thick] (s.corner 2) to (s.corner 8);
    \draw[thick] (s.corner 2) to (s.corner 7);
    \draw[thick] (s.corner 2) to (s.corner 6);
    \draw[thick] (s.corner 3) to (s.corner 6);

    \draw (3.2,0) node[inner sep=-1pt,dotted] (et) {$\circ$};
    \draw (3.6,0) node[dotted] {1};
    \draw (-2.2,2.2) node[inner sep=-1pt,dotted] (to) {$\circ$};
    \draw (-2.5,2.5) node[dotted] {2};
    \draw (0,3.2) node[inner sep=-1pt,dotted] (tre) {$\circ$};
    \draw (0,3.6) node[dotted] {3};
    \draw (2.2,2.2) node[inner sep=-1pt,dotted] (fire) {$\circ$};
    \draw (2.5,2.5) node[dotted] {4};
    \draw (-2.2,-2.2) node[inner sep=-1pt,dotted] (fem) {$\circ$};
    \draw (-2.5,-2.5) node[dotted] {5};
    \draw (-3.2,0) node[inner sep=-1pt,dotted] (seks) {$\circ$};
    \draw (-3.6,0) node[dotted] {6};
    \draw (2.2,-2.2) node[inner sep=-1pt,dotted] (syv) {$\circ$};
    \draw (2.5,-2.5) node[dotted] {7};
    \draw (-1.5,1.2) node[inner sep=-1pt,dotted] (a) {$\circ$};
    \draw (1.2,1.2) node[inner sep=-1pt,dotted] (b) {$\circ$};
    \draw (-1.5,-1.5) node[inner sep=-1pt,dotted] (c) {$\circ$};
    \draw (-1.2,-1.6) node[dotted] {*};
    \draw (1.2,-1.5) node[inner sep=-1pt,dotted] (d) {$\circ$};

    \draw[dotted] (et) to (b);
    \draw[dotted] (tre) to (b);
    \draw[dotted] (fire) to (b);
    \draw[dotted] (d) to (b);
    \draw[dotted] (to) to (a);
    \draw[dotted] (a) to (d);
    \draw[dotted] (d) to (syv);
    \draw[dotted] (fem) to (c);
    \draw[dotted] (seks) to (c);
    \draw[dotted] (a) to (c);

  \end{tikzpicture} 
  \caption{The correspondence between $R'$-enriched rooted trees and labelled
    Ptolemy diagrams with base edge.}
\label{fig:correspondence}
\end{figure}
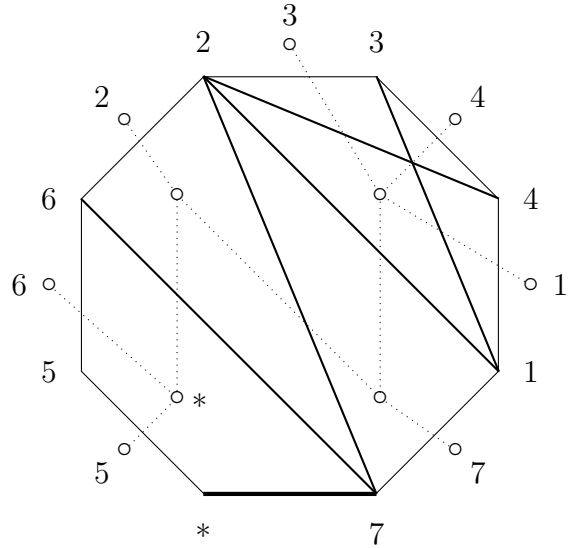

In our situation, we set $R=Y+C_{\geq 3}+C_{\geq 4}$ where $C_{\geq
  k}$ denotes the species of cycles with at least $k$ vertices.  The
derivative of $R$ is
\[
R'=1+\mathcal L_{\geq 2}+\mathcal L_{\geq 3},
\]
where $\mathcal L_{\geq k}$ denotes the species of lists with at
least $k$ elements.  We now see that $B_{R'}$ is isomorphic (in the
sense of \cite[def.\ 12, sec.\ 1.2]{BLL}) to the combinatorial
species of Ptolemy diagrams with a distinguished base edge and labels
on all vertices except the counterclockwise first on the base edge.
Namely, a Ptolemy diagram can be regarded as an $R'$-enriched rooted
tree as follows: the region attached to the distinguished base edge
corresponds to the root and the other regions to the internal
vertices of the tree, i.e., vertices which are not leaves, see
Figure~\ref{fig:correspondence}.  Note that the degenerate Ptolemy
diagram, consisting of the base edge only, carries one label.  This
corresponds to the tree consisting of one isolated vertex, which is
also labelled -- despite being the root of the tree.

Let us informally explain the meaning of the three summands in $R'$:
the first summand, $1$, applies if a vertex is a leaf and thus has no
successor.  The second summand, $\mathcal L_{\geq 2}$, applies if a
vertex corresponds to a region that is of type (ii) in the
decomposition of Proposition \ref{pro:decomposition}, i.e., an empty
cell, in which case the vertex must have at least two successors.
Finally, the third summand, $\mathcal L_{\geq 3}$, applies if a
vertex corresponds to a region that is of type (iii) in Proposition
\ref{pro:decomposition}, in which case the vertex must have at least
three successors.  In the latter two cases the species of lists
imposes an ordering onto the successors of the vertex.

In a similar manner we can see that $b_R$ is the species of Ptolemy
diagrams up to rotation and labels on \emph{all} vertices.  Here,
enriching the inner vertices with the species of cycles imposes a
cyclic ordering on the neighbours of each vertex.

We can now state the announced tool.  We reproduce it here in a
slightly simplified form; it is the special case of Theorem 4.1.7 in
\cite{BLL} obtained by setting $X=1$.  In this special case we
additionally have to require $\# R_0[\{1,2\}]=0$ to ensure
well-definedness of the species involved.
\begin{Proposition}\label{pro:dissymmetry}
  Let $R_0$ be a combinatorial species such that $\#
  R_0[\emptyset]=\# R_0[\{1\}]=\# R_0[\{1,2\}]=0$ and let $R=R_0+Y$.
  Then the combinatorial species $b_R$ of $R$-enriched trees and the
  combinatorial species of $R'$-enriched rooted trees $B_{R'}$ are
  related as follows:
  \begin{equation*}
    b_R + B_{R'}^2 = (E_2 + R_0)\circ B_{R'} + Y\cdot B_{R'}.
  \end{equation*}
\end{Proposition}

As far as the enumeration of \emph{labelled} structures is concerned
this proposition is not very interesting.  Namely, it follows
directly from the definition of the derivative of a species that
$B_{R'}$ is the derivative of $b_R$: the correspondence is
accomplished by making the root into another labelled vertex.  In
particular, the number of labelled Ptolemy diagrams up to rotation
with $N+1$ vertices (and $N+1$ labels) equals the number of labelled
Ptolemy diagrams with distinguished base edge with $N+1$ vertices
(and $N$ labels) and is given by $N!$ times the $N$-th coefficient of
$\P(y)$.

However, the proposition enables us to determine also the (ordinary)
generating function of unlabelled Ptolemy diagrams up to rotation.
In the jargon of combinatorial species this is the \Dfn{isomorphism
  type generating function} $\widetilde b_R(y)$ of the species $b_R$
with the specific value of $R$ used above.  In general, the
\Dfn{isomorphism type generating function} of a species $\F$ is
denoted $\widetilde \F(y)$ and we have the usual rules
$\widetilde{(\F+\G)}(y)=\widetilde \F(y)+\widetilde \G(y)$ and
$\widetilde{(\F\G)}(y)=\widetilde \F(y)\widetilde \G(y)$.  To compute
$\widetilde b_R(y)$ we additionally need to use cycle indicator
series.  We collect the facts significant for us in the following
lemma.
\begin{Lemma}
  Let $\F$ be a combinatorial species and $Z_\F$ its cycle indicator
  series.  Then the generating function for the isomorphism types of
  $\F$ is given by
  $$
  \widetilde \F=Z_\F(y, y^2, y^3, \dots) \quad\text{(see \cite[thm.\ 8,
    sec.\ 1.2]{BLL})}.
  $$
  Moreover, let $\G$ be another species, satisfying $\#
  \G[\emptyset]=0$.  Then the generating function for the isomorphism
  types of $\F\circ \G$ is given by
  $$
  \widetilde{\F\circ \G}=Z_\F\big(\widetilde \G(y),\widetilde
  \G(y^2),\widetilde \G(y^3),\dots\big) \quad\text{(see \cite[thm.\ 2,
    sec.\ 1.4]{BLL})}.
  $$
  The cycle indicator series of the species of cycles $C$ is given by
  $$
  Z_C(p_1, p_2, \dots) = \sum_{d\geq 1}
  \frac{\phi(d)}{d}\log\left(\frac{1}{1-p_d}\right),
  $$
  where $\phi$ is Euler's totient (see \cite[eq.\ (18), sec.\
  1.4]{BLL}). 

  The cycle indicator series of the two element set $E_2$ (which
  coincides with the $2$-cycle $C_2$) is given by
  $$
  Z_{E_2}(p_1, p_2, \dots) = \frac{1}{2}(p_1^2 + p_2)
  \quad\text{(see \cite[table 5, app.\ 2]{BLL})}.
  $$
  The cycle indicator series of the $3$-cycle $C_3$ is given by
  $$
  Z_{C_3}(p_1, p_2, \dots) = \frac{1}{3}(p_1^3 + 2p_3)
  \quad\text{(see \cite[table 5, app.\ 2]{BLL})}.
  $$
\end{Lemma}

Note that, since $\P(y)$ is algebraic, $\widetilde\P(y)=\P(y)$.
Putting all the bits together, we find:
\begin{Proposition}
\label{pro:Martin}
  The generating function for Ptolemy diagrams up to rotation is
  \begin{multline*}
    2\sum_{d\geq 1}\frac{\phi(d)}{d} \log\left(\frac{1}{1-\P(y^d)}\right)\\
    - \frac{1}{2}\big(3\P(y)^2+\P(y^2)\big)%
    - \frac{1}{3}\big(\P(y)^3+2\P(y^3)\big)%
    - 2\P(y) + y\P(y),
  \end{multline*}
  where $\P(y)$ is the generating function for Ptolemy diagrams with
  a distinguished base edge, and $\phi(d)$ is Euler's totient.
\end{Proposition}
The first few coefficients are given in the introduction.
\begin{proof}
  We use Proposition~\ref{pro:dissymmetry} with $R_0=C_{\geq
    3}+C_{\geq 4}$.  Since (formally) $E_2 + R_0 = C_{k\geq 2} +
  C_{k\geq 4}=2C - 2Y - E_2 - C_3$,
  \begin{align*}
    Z_{E_2 + R_0} = %
    2\sum_{d\geq 1}\frac{\phi(d)}{d}\log\left(\frac{1}{1-p_d}\right)%
    -2p_1-\frac{1}{2}(p_1^2 + p_2)-\frac{1}{3}(p_1^3 + 2p_3).
  \end{align*}

  Since $\P(y)$ is algebraic, we have $\widetilde{B_{R'}}=\P(y)$ and
  therefore
  \begin{align*}
    \widetilde b_R(y) =& Z_{E_2 + R_0}\big(\P(y), \P(y^2),\dots\big)
    +
    y\P(y) - \P(y)^2\\
    =&2\sum_{d\geq 1}\frac{\phi(d)}{d}%
    \log\left(\frac{1}{1-\P(y^d)}\right)\\%
    &-2\P(y)-\frac{1}{2}\big(\P(y)^2 +
    \P(y^2)\big)-\frac{1}{3}\big(\P(y)^3 + 2\P(y^3)\big)+ y\P(y) -
    \P(y)^2,
  \end{align*}
  which is equivalent to the claim.
\end{proof}

\end{document}